\documentclass[11pt,a4paper]{article}

\usepackage{amssymb,amsmath,amsthm,color,verbatim}
\usepackage{hyperref}
\usepackage[T1]{fontenc}

\newtheorem{thm}{Theorem}[section]
\theoremstyle{definition}
\newtheorem{de}[thm]{Definition}
\newtheorem{ex}[thm]{Example}
\newtheorem{lem}[thm]{Lemma}
\newtheorem{co}[thm]{Corollary}
\newtheorem{re}[thm]{Remark}

\newtheorem{pro}[thm]{Proposition}

\newtheorem*{claim}{Claim}
{
\newcommand{\ba}{\begin{array}} \newcommand{\ea}{\end{array}}

\newcommand{\N}{\mathbb{N}}                                            
\newcommand{\R}{\mathbb{R}}                                         
\newcommand{\eps}{\varepsilon}                                         
\DeclareMathOperator{\diam}{diam}

\DeclareMathOperator{\sep}{sep}
\DeclareMathOperator{\CAT}{CAT}                                        
     
\DeclareMathOperator{\BW}{(BW)}
\DeclareMathOperator{\UBW}{(UBW)}

\begin{document}

\title{A uniform betweenness property in metric spaces and its role in the 
quantitative analysis of the ``Lion-Man'' game}

\author{Ulrich Kohlenbach$^{a}$, Genaro L\'{o}pez-Acedo$^{b}$, Adriana Nicolae$^{c}$}
\date{}
\maketitle

\begin{center}
{\scriptsize
$^{a}$Department of Mathematics, Technische Universit\" at Darmstadt, Schlossgartenstra\ss{}e 7, 64289 Darmstadt, Germany
\ \\
$^{b}$Department of Mathematical Analysis - IMUS, University of Seville, C/ Tarfia s/n, 41012 Seville, Spain
\ \\
$^{c}$Department of Mathematics, Babe\c s-Bolyai University, Kog\u alniceanu 1, 400084 Cluj-Napoca, Romania \\
\ \\
E-mail addresses: kohlenbach@mathematik.tu-darmstadt.de (U. Kohlenbach), glopez@us.es (G. L\'{o}pez-Acedo), anicolae@math.ubbcluj.ro (A. Nicolae)
}
\end{center}

\begin{abstract}
In this paper we analyze, based on an interplay between ideas and techniques from logic and geometric analysis, a pursuit-evasion game. More precisely, we focus on a uniform betweenness property and use it in the study of a discrete lion and man game with an $\eps$-capture criterion. In particular, we prove that in uniformly convex bounded domains the lion always wins and, using ideas stemming from proof mining, we extract a uniform rate of convergence for the successive distances between the lion and the man. As a byproduct of our analysis, we study the relation among different convexity properties in the setting of geodesic spaces.  \\

\noindent {\em MSC: 91A24, 49N75 (primary), 53C23, 03F10 (secondary)} 
\\

\noindent {\em Keywords}: Lion and man game, \and rate of convergence, \and uniform convexity, \and betweenness property.
\end{abstract}

%
\section{Introduction}
The lion and man problem, which goes back to R. Rado, is one of the most challenging pursuit-evasion games. In the inspiring book {\it Littlewood's Miscellany} \cite{Li86} it is described as follows:  
\begin{quotation}
\sl A lion and a man in a closed circular arena have equal maximum speeds. What tactics should the lion employ to be sure of his meal? 
\end{quotation}
A detailed discussion of the solution to this problem can be found in \cite{Cr64,Li86,Na07}. Very similar problems have appeared under different names in the literature (e.g. the robot and the rabbit \cite{H69} or the cop and the robber \cite{AF84}).

The analysis of the lion and man game is closely tied to the geometric structure of the domain where the game is played. This fact, as well as the potential applications in different fields such as robotics \cite{BH10},  biology \cite{BBH08},  and random processes \cite{BBK13,BBK14}, have given rise to several variants of this game, both continuous \cite{BLW12} and discrete (the discrete version is attributed to D. Gales, see \cite{S01} for more details). Such games involve one or more evaders  in a fixed domain being hunted by one or more pursuers who win the game if certain appropriate capture criteria are satisfied. Such criteria may be physical capture (the pursuers move to the location of the evaders) or $\eps$-capture \cite{AGR92} (the pursuers get within a distance less than $\eps$ to the evaders). Another very important feature in the game is the possibility of assuming different degrees of freedom in the movement of the lion.  

Here we focus on a discrete-time equal-speed game   with an $\eps$-capture criterion. The domain $X$ of our game is a  geodesic space. Initially, the lion and the man are located at two points in $X$, $L_0$ and $M_0$, respectively. One fixes a positive upper bound $D$ on the distance the lion and the man may jump. After $n$ steps, the lion moves from the point $L_n$ to the point $L_{n+1}$ along a geodesic from $L_{n}$ to $M_{n}$, that is $d(L_n,M_n)= d(L_n,L_{n+1}) +d(L_{n+1},M_n)$,  such that its distance to $L_{n}$ equals $\min\{D,d(L_n,M_n)\}$. The man  moves from the point $M_{n}$ to any point $M_{n+1} \in X$ which is within distance $D$. Given a metric space, we say that the lion wins if $\lim_{n\to\infty} d(L_{n+1},M_n) = 0$ for any pair of sequences $( L_{n}),  (M_{n})$ that satisfy the previous metric conditions for any $D>0$. Otherwise the man wins. When we refer in the sequel to the Lion-Man game, we will always mean the game we have just described. 

In \cite{AleBisGhr10}, a similar game is introduced for the particular case of uniquely geodesic spaces, so that the movement of the lion is completely determined by the movement of the man. The authors prove that in the setting of nonpositively curved bounded domains the lion always wins. Further advances in this problem were made in \cite{LNP1} (see also \cite{Bac12,Yuf19}), where a characterization of compactness of the domain in terms of the success of the lion was obtained in complete, locally compact, strongly convex  geodesic spaces. The main ingredient in the proof of this result is the fact that strongly convex spaces   satisfy the betweenness property. However, none of these results provides any information on the speed of convergence towards $0$ of the sequence $d(L_{n+1},M_n)$.

Our aim in this paper is to introduce and study a quantitative uniform version of the aforementioned betweenness property  and to use it in the analysis of the Lion-Man game. This allows us to weaken the topological and geometric hypotheses that ensure the success of the lion and to give a rate of convergence for the sequence $d(L_{n+1},M_n)$ that only depends on a modulus quantifying the uniform betweenness property. The ideas that led to our results have their roots in proof mining. By `proof mining' we mean the logical analysis, using proof-theoretic tools, of mathematical proofs with the aim of extracting relevant information hidden in the proofs. This new information can be both of quantitative nature, such as algorithms and effective bounds, as well as of qualitative nature, such as uniformities in the bounds or weakening of the premises. A comprehensive reference for proof mining is the book \cite{K08} (see 
also \cite{K19} for a recent survey). 
 
The organization of the paper is as follows. In Section \ref{section-rate} we prove the main result  which relies on the betweenness property. This property (see Definition \ref{definition-betweenness}) holds e.g. in all strictly convex normed spaces, but also in a wide class of geodesic spaces as we will show in Sections \ref{section-UBTW-uniquely-geod} and \ref{section-UBTW-nonuniquely-geod}. Betweenness relations have already been considered in very early works such as \cite{HK17,DW}.  We introduce a quantitative uniform variant of the betweenness property (see Definition \ref{def-uniform-btw}), and associate to this definition a modulus of uniform betweenness. All these definitions are purely metric, what led to state our main result, Theorem \ref{thm-main}, in a purely metric setting. More precisely, we prove  that under the assumption of boundedness and uniform betweenness for the domain, the lion always wins. The result applies in particular for all uniformly convex normed spaces, $\CAT(\kappa)$ spaces (of sufficiently small diameter for $\kappa > 0$), or compact uniquely geodesic spaces satisfying the betweenness property. Consequently, we notably weaken and unify previously known geometric conditions that were imposed on the domain in order to guarantee the success of the lion. Moreover, Corollary \ref{co-main} provides a rate of convergence for the sequence $d(L_{n+1},M_n)$ towards $0$, and hence gives an explicit bound on the number of steps to be taken for an $\eps$-capture.

In Sections \ref{section-UBTW-uniquely-geod} and \ref{section-UBTW-nonuniquely-geod} we analyze the connection of the betweenness property with other convexity properties. In Section  \ref{section-UBTW-uniquely-geod}  we consider the case of uniquely geodesic spaces. Although the existence of unique geodesics between any two given points in a geodesic space is a widely known and well-understood condition, here we consider a quantitative uniform version thereof (see Definition \ref{def-UC}) and study its connection with other convexity properties. In Theorem \ref{thm-UC-UU} we prove that uniformly convex geodesic spaces admitting a monotone modulus of uniform convexity $\eta$ are uniformly uniquely geodesic and one can define a modulus of uniform uniqueness in terms of $\eta$. Actually, in normed vector spaces, uniform convexity is equivalent to uniform uniqueness of geodesics, but in general these two concepts are different. This distinction in nonlinear settings is an interesting feature of uniform uniqueness which could motivate the further study of its relevance.  
In Theorem \ref{thm-UU-UBTW} we show that the uniform betweenness property holds in uniformly uniquely geodesic spaces whose distance function satisfies a convexity condition, and that a modulus of uniform uniqueness generates a modulus of uniform betweenness.  Section \ref{section-UBTW-nonuniquely-geod} is devoted to  provide two particular instances of nonuniquely geodesic spaces  where the uniform betweenness property holds: geodesic Ptolemy spaces (which can be nonuniquely geodesic) and a certain nonstrictly convex normed space of dimension 3 considered in \cite{DW}.  In both cases a modulus of uniform betweenness is computed.  The last section contains a general discussion on the proof mining techniques used to develop our quantitative analysis.

\section{A rate of convergence for the Lion-Man game}\label{section-rate}
The main goal of this section is to introduce and study a new uniform concept of betweenness for metric spaces and to establish its crucial role for 
the Lion-Man 
game. In particular, the concept of a `modulus of uniform betweenness' 
will be used for a quantitative analysis of the Lion-Man game. This uniform betweenness property will be further studied in Sections \ref{section-UBTW-uniquely-geod} and \ref{section-UBTW-nonuniquely-geod}.
\subsection{Betweenness and uniform betweenness}
We start by introducing two important geometric concepts.
\begin{de}[Condition 1 in \cite{DW}]\label{definition-betweenness}
A metric space $(X,d)$ satisfies the {\it betweenness property} $\BW$ if for 
any pairwise distinct points $x,y,z,w\in X$ the following holds
\begin{eqnarray} \nonumber
\left.\begin{array}{l}
d(x,y)+d(y,z)\le d(x,z) \\
d(y,z)+d(z,w)\le d(y,w)
\end{array}
\right\}
& \Rightarrow & d(x,z)+d(z,w)\le d(x,w).
\end{eqnarray}
\end{de}
Postulates for betweenness relations and their relevance with other convexity 
conditions already appeared in very early works such as \cite{HK17,DW}. 
The property given in Definition \ref{definition-betweenness} was studied 
in connection to the geometry of geodesic metric spaces in \cite{Pap05,Nic13,LNP1} (we refer to Section \ref{section-UBTW-uniquely-geod} for some basic definitions on geodesic metric spaces). Proposition 3.4 in \cite{Nic13} shows in particular that $\BW$ holds in every uniquely geodesic space $X$ satisfying the following convexity condition
\begin{equation}\label{eq-metric-conv}
d(z,(1-t)x+ty) \le (1-t)d(z,x) + td(z,y),
\end{equation}
for all $x,y,z \in X$ and all $t \in [0,1]$. In other words, given any $z \in X$, the function $d(z,\cdot)$ is convex. This condition holds e.g. in any strictly convex normed space, any geodesic space that is nonpositively curved in the sense of Busemann or in any CAT$(\kappa)$ space (of diameter smaller than $\pi/(2\sqrt{\kappa})$ if $\kappa > 0$).

The $\BW$ condition is also satisfied by some nonstrictly convex Banach spaces (see \cite{DW}) but 
e.g. not by $(\R^2,\|\cdot\|_{\infty})$ or $(\R^2,\|\cdot\|_1).$ 

We next present a stronger uniform version of $\BW$
which is equivalent to the ordinary one when $X$ is compact but which, as 
we will see in the next sections, also holds in many noncompact situations e.g. 
in uniformly convex Banach spaces.

Let $(X,d)$ be a metric space. If $A$ is a nonempty and bounded subset of $X$, the {\it diameter} of $A$ is 
\[\diam(A) = \sup\{d(a,a'): a,a' \in A\},\]
and the {\it separation} of $A$ is 
\[\sep(A) = \inf\{d(a,a'): a,a' \in A, a \ne a'\}.\] 

\begin{de}\label{def-uniform-btw}
A metric space $(X,d)$ satisfies the {\it uniform betweenness property} 
$\UBW$ if for all $a,b,\eps>0$ there exists $\theta>0$ such that for all 
$x,y,z,w\in X$ we have
\begin{eqnarray} \nonumber
\left.\begin{array}{l}
\sep\{x,y,z,w\}\ge a\\
\diam\{x,y,z,w\} \le b\\
d(x,y)+d(y,z)\le d(x,z)+\theta\\
d(y,z)+d(z,w)\le d(y,w)+ \theta
\end{array}
\right\}
& \Rightarrow & d(x,z)+d(z,w)\le d(x,w)+\eps.
\end{eqnarray}
Any function $\Theta:(0,\infty)^3\to (0,\infty)$ which provides for given 
$\eps,a,b>0$ such a 
$\theta = \Theta(\eps,a,b)$ is called a {\it modulus of uniform betweenness}.
\end{de}

It is easy to see that every metric space satisfying $\UBW$ also satisfies $\BW$. We include the proof of the converse relation in the presence of compactness.

\begin{pro}\label{prop-compact-BW}
Compact metric spaces with $\BW$ satisfy $\UBW$ as well.
\end{pro}
\begin{proof}
Although this fact is almost straightforward, we include its proof because $\UBW$ plays an essential role in this work. We argue by contradiction. Suppose that $(X,d)$ is a compact metric space with $\BW$, but without $\UBW$. Then there exist $\eps,a > 0$ such that for all $n \in \mathbb{N}$ we can find points $x_n, y_n, z_n, w_n \in X$ with $\sep\{x_n, y_n, z_n, w_n\} \ge a$ satisfying 
\[d(x_n,y_n)+d(y_n,z_n)\le d(x_n,z_n)+\frac{1}{n}, \quad d(y_n,z_n)+d(z_n,w_n)\le d(y_n,w_n) + \frac{1}{n}\]
and 
\begin{equation} \label{eq1-cmp-UBW}
 d(x_n,z_n)+d(z_n,w_n) > d(x_n,w_n)+\eps. 
\end{equation}
By compactness, we may assume that there exist $x,y,z,w \in X$ such that $x_n \to x$, $y_n \to y$, $z_n \to z$ and $w_n \to w$. Then $x,y,z,w$ are pairwise distinct and 
\[d(x,y)+d(y,z)\le d(x,z), \quad d(y,z)+d(z,w)\le d(y,w).\]
By $\BW$, we obtain $d(x,z)+d(z,w)\le d(x,w)$. This contradicts \eqref{eq1-cmp-UBW}.
\end{proof}

In the case of normed spaces $(X,\|\cdot\|)$, 
it was proved in \cite{DW} that the betweenness property
$\BW$ is equivalent to the following property: for all $x,y,z\in X$,
\[ \BW': \quad \| x\|=\| y\|=\| z\|=\left\|\frac{x+y}{2}\right\|=
\left\|\frac{y+z}{2}\right\|=1\quad \Rightarrow \quad \| x+y+z\|=3. \]
$\BW'$ also has an obvious uniformization:
for all $\eps>0$ there exists $\delta>0$ such that for all $x,y,z\in X$,
\begin{eqnarray}\nonumber
\UBW':\quad
\left.\begin{array}{l}
\| x\|=\| y\|=\| z\|=1\\[2mm]
\displaystyle \min\left\{\left\|\frac{x+y}{2}\right\|,\left\|\frac{y+z}{2}\right\|\right\}\ge 1-\delta
\end{array}
\right\}
& \Rightarrow & \| x+y+z\|\ge 3-\eps
\end{eqnarray}
together with the corresponding concept of a modulus $\delta:(0,\infty)\to
(0,\infty)$ such that $\delta(\eps)$ satisfies $\UBW'$. 

We provide below a uniform quantitative analysis of the proof
of the equivalence of $\BW$ and $\BW'$ given in \cite{DW}. Before stating this result, we include the following simple property.
\begin{lem}\label{lemma-norm-comb}
Let $(X,\Vert \cdot \Vert)$ be a normed space, $\eps, \mu,\lambda\ge 0$, and $x,y \in X$ with $\| x+y\|\ge 
\| x\| +\| y\| -\eps$. Then $\| \lambda x+\mu y\|\ge \lambda\| x\|+\mu\| y\| -\max\{ \lambda,\mu\}\eps$.
\end{lem}

\begin{pro} \label{prop-modulus-UBW'}
Let $(X,\|\cdot\|)$ be a normed space. Then $X$ satisfies 
$\UBW$ if and only if it satisfies $\UBW'$. Moreover, respective moduli can be transformed 
into each other by the transformations 
\[\Theta(\eps,a,b):=2a\cdot \delta\left(\frac{\eps}{2b}\right) \quad \text{and} \quad
\delta(\eps):=\frac{1}{2}\min\left\{ \Theta\left(\frac{\eps}{2},
\frac{1}{2},3\right), 
\frac{1}{2},\frac{\eps}{2}\right\}. \]
\end{pro}
\begin{proof} Let first $\delta$ be a modulus for $X$ satisfying $\UBW'$.
Suppose $x,y,z,w\in X$ such that $\sep\{ x,y, z,w\}\ge a>0$. Take 
$b\ge\diam\{ x,y,z,w\}$ and assume that for $\Theta$ as defined in the 
statement we have
\[\| x-y\|+\| y-z\| \le \| x-z\| +\Theta, \quad \| y-z\|+\| z-w\|\le \| y-w\| +\Theta. \]
Then taking $\tilde{u}:=y-x$, $\tilde{v}:=z-y$, $\tilde{w}:=w-z$,
\[ \|\tilde{u}\|+\| \tilde{v}\| \le\| \tilde{u}+\tilde{v}\|
+\Theta, \quad \|\tilde{v}\|+\| \tilde{w}\| \le\| \tilde{v}+\tilde{w}\|+\Theta. \] 
Hence, using Lemma \ref{lemma-norm-comb} and the fact that $\| \tilde{u}\|, \|\tilde{v}\|, \| \tilde{w}\|\ge a$, we get
\[ \| u\|+\| v\| \le \| u+v\| +2\delta(\tilde{\eps}), \quad
\| v\|+\| w\| \le \| v+w\| +2\delta(\tilde{\eps}), \]
where 
\[u:=\frac{\tilde{u}}{\| \tilde{u}\|}, \quad v:=\frac{\tilde{v}}{\| \tilde{v}\|}, \quad w:=\frac{\tilde{w}}{\| \tilde{w}\|} \quad \text{and} \quad\tilde{\eps}:=\frac{\eps}{2b}.\]
Thus,
\[ \min\left\{\left\| \frac{u+v}{2}\right\|,\left\| \frac{v+w}{2}\right\|\right\}\ge 1-\delta(\tilde{\eps})\] and so, by $\UBW'$, 
$\| u+v+w\| \ge \| u\| +\| v\| +\| w\| -\tilde{\eps}.$ In particular,
\begin{equation}\label{eq2}
\| u\| +\| v\| \le \| u+v\| +\tilde{\eps} \quad \text{and} \quad \| u+v+w\|\ge \| u+v\| +\| w\| -\tilde{\eps}.    
\end{equation}
Hence for $\alpha:=\| y-x\|$, $\beta:=\| z-y\|$, $\gamma:=\|w-z\|$ we have $\alpha,\beta,\gamma\le b$ and 
(w.l.o.g. $\beta\ge \alpha$)
\begin{align*}
\| x-w\| & =\|\tilde{u}+\tilde{v}+\tilde{w}\|=\|\alpha u+\beta v+\gamma w\| \ge  \| \beta(u+v)+\gamma w\| -(\beta-\alpha)\| u\| \\
&\ge \beta\|u+v\|+\gamma \|w\| -(\beta-\alpha)\| u\|-b\cdot\tilde{\eps} \quad \text{by \eqref{eq2} and Lemma \ref{lemma-norm-comb}} \\
&\ge \beta(\| u\| +\| v\| -\tilde{\eps})+\gamma\| w\| -\beta \| u\| +\alpha \| u\| -b\cdot\tilde{\eps} 
\quad \text{by \eqref{eq2}} \\ 
&\ge \alpha\| u\| +\beta\| v\| +\gamma\| w\| -2b\tilde{\eps} =
\| \tilde{u}\| +\|\tilde{v}\|+\| \tilde{w}\| -\eps\ge \| x-z\| +\| z-w\| -\eps. 
\end{align*}

In the other direction, let $\Theta$ be a modulus for $\UBW$ (which by
switching to $\Theta'(\eps,a,b):=\min\{ \eps,1/2,\Theta(\eps,a,b)\}$ 
we may assume to satisfy $\Theta(\eps,a,b)\le \min\{\eps,1/2\}$). 
Define 
\[\delta(\eps):=\frac{1}{2}\Theta\left(\frac{\eps}{2},\frac{1}{2},3\right)\]
and assume 
that, given $\eps>0$, the points $u,v,w\in X$ satisfy the assumption 
of $\UBW'$, i.e. 
\[ \| u\|=\| v\|=\| w\|=1, \quad \| u+v\|\ge \| u\| +\| v\|-
\Theta\left(\frac{\eps}{2},\frac{1}{2},3\right), \quad \| v+w\|\ge \| v\| +\| w\|-
\Theta\left(\frac{\eps}{2},\frac{1}{2},3\right).\] 
Denoting $x:=0$, $y:=u$, $z:=u+v$, $t:=u+v+w$, the previous two inequalities become
\[ \| z-x\|\ge \| y-x\| +\| z-y\|-
\Theta\left(\frac{\eps}{2},\frac{1}{2},3\right), \quad \|t-y\|\ge \| z-y\| +\| t-z\|-
\Theta\left(\frac{\eps}{2},\frac{1}{2},3\right).\] 
As $\sep\{ x,y,z,t\} \ge 1/2$ and $\diam\{ x,y,z,t\}\le 3$, by $\UBW$,
\[ \| u+v\| +\| w\| =\| z-x\| +\| t-z\|\le \| t-x\| +\frac{\eps}{2} =
\| u+v+w\| +\frac{\eps}{2} \] and so 
\[ 3=\| u\|+\| v\| +\| w\| \le \| u+v\| +\| w\|+
\Theta\left(\frac{\eps}{2},\frac{1}{2},3\right)\le \| u+v+w\| +
\Theta\left(\frac{\eps}{2},\frac{1}{2},3\right)+\frac{\eps}{2}\le 
\| u+v+w\| +\eps. \]
\end{proof}
\subsection{A rate of convergence for the Lion-Man game}
We analyze in the sequel the Lion-Man game in the 
context of general metric spaces which satisfy $\UBW$. We recall first the exact definition of the game in this 
abstract setting. Let $(X,d)$ be 
a metric space. By a {\it Lion-Man game} with speed $D>0$ 
we mean a pair $\langle (M_n),(L_n) 
\rangle$ of sequences 
in $X$ such that for all $n\in\N$ 
\[ \left\{ \ba{l} 
d(M_n,M_{n+1})\le D, \ d(L_{n+1},L_n)+d(L_{n+1},M_n)=d(L_n,M_n) \ \mbox{and}
\\[2mm] d(L_n,L_{n+1})=\min\{ D,d(L_n,M_n)\}. \ea\right. \] 
We say that the lion wins if the sequence $(d(L_{n+1},M_n))$ converges to $0$. Otherwise the man wins. 

The main result of this paper shows that the lion always wins if $X$ is bounded 
and satisfies $\UBW$. Moreover, we provide an 
explicit rate of convergence for the sequence $(d(L_{n+1},M_n))$ towards $0$ 
which only depends on $D,$ $\varepsilon,$ an upper bound $b\ge \diam(X)$ 
and a modulus of uniform betweenness $\Theta$ for $X.$ 

Before giving our main result, we recall the following property of bounded nonincreasing real sequences which follows from Proposition 2.27
and Remark 2.29 in \cite{K08}.
\begin{lem}[Kohlenbach \cite{K08}] \label{lemma-mon-seq}
Let $K>0$ and $(a_n)$ be a nonincreasing sequence in $[0,K]$. Then 
\[\forall \tau > 0\; \forall g:\N\to\N\; \exists I \le \tilde{g}^{\left(\left\lceil\frac{K}{\tau}\right\rceil\right)}(0) \; \forall n,m\in[0,g(I)]\; \left(|a_{I+n} - a_{I+m}| \le \tau\right),\]
where $\tilde{g} := \text{Id}+g$ and $\tilde{g}^{n+1}(0) := \tilde{g}(\tilde{g}^n(0))$, $\tilde{g}^0(0) := 0$.
\end{lem}

Denote $D_n = d(L_n,M_n)$, $n \in \N$. Note first that if $D_n \ge D$, then 
\begin{equation}\label{eq-Dn-decr}
D_{n+1} \le d(L_{n+1},M_n) + d(M_n,M_{n+1}) = D_n - D + d(M_n,M_{n+1}) \le D_n.
\end{equation}
Thus, if $D_n \ge D$ for all $n \in \N$, then $(D_n)$ is nonincreasing.

We can distinguish two mutually exclusive situations when the lion wins:
\begin{itemize}
\item[(1)] there exists $n_0 \in \N$ such that $D_{n_0} < D$. 
\item[(2)] $D_n \ge D$ for all $n \in \N$ and $\lim_{n \to \infty}D_n = D$.
\end{itemize}

In the following, let $(X,d)$ be a metric 
space which satisfies $\UBW$ with
a modulus of uniform betweenness 
$\Theta:(0,\infty)^3\to (0,\infty).$ 
We define $\Theta(\eps):=\Theta(\eps,\eps,b)$ and assume w.l.o.g. that 
$\Theta(\eps)\le \eps$ for all $\eps>0$ (otherwise take 
$\Theta'(\eps,b):=\min\{\eps,\Theta(\eps,b)\}$). Let $b\ge \diam(X)>0$ and $D>0.$ Take $N\in\N$ such that 
\begin{equation}\label{eq-diam-A}
b+1 < ND,
\end{equation}
e.g. $N:=\left\lceil \frac{b+1}{D}\right\rceil +1.$

\begin{thm}\label{thm-main}
Let $\langle (M_n),(L_n)\rangle$ be a Lion-Man game in $X$ with speed $D.$
Then 
\[ \forall \eps>0\,\forall n\ge\Omega_{D,b,\Theta}(\eps) \ 
\left( d(L_n,M_n)<D+\eps\right), \]
where 
\[ \Omega_{D,b,\Theta}(\eps):=N+N\left\lceil \frac{b}{\Theta^{(N)}(\alpha)}
\right\rceil \ \mbox{with} \]
\begin{equation}\label{thm-main-eq-eps} 
0<\alpha\le\min\left\{ \frac{1}{N},\frac{\eps}{2},\frac{D}{2}\right\}. 
\end{equation} 
\end{thm}
\begin{co}\label{co-main} Under the same assumptions:
\[ \forall \eps>0\,\forall n\ge\Omega_{D,b,\Theta}(\eps) \ 
\left( d(L_{n+1},M_n)<\eps\right). \]
\end{co}
\begin{proof} The claim follows from the theorem since 
$d(L_{n+1},M_n)=\max\{0, d(L_n,M_n)-D\}.$
\end{proof}

\begin{re}
Using Proposition \ref{prop-compact-BW}, we obtain as an immediate consequence one of the implications proved in \cite[Theorem 4.2]{LNP1}, namely that the lion always wins the Lion-Man game played in a compact geodesic space that satisfies $\BW$. Even more, the condition that the space is uniquely geodesic imposed in \cite[Theorem 4.2]{LNP1} is no longer assumed (see also \cite{Yuf19}).
\end{re}

\begin{proof}[Proof of Theorem \ref{thm-main-eq-eps}]
Let $\eps > 0$ and $0<\alpha\le \min\left\{1/N,\eps/2,D/2 \right\}.$ 
We use the notation introduced above. For simplicity, denote
\[\omega = \Omega_{D,b,\Theta}(\eps).\] 
Suppose first that there exists $n_0 \in \mathbb{N}$ such that $D_{n_0} < D$. If $n_0 \le \omega$, then for all $n \ge n_0$, $D_n \le D < D+\eps$ and the conclusion holds. 

So we only need to consider the following two cases:
\begin{itemize}
\item[(i)] there exists $n_0 > \omega$ with $D_{n_0} < D$ and $D_n \ge D$ for all $n \le n_0 - 1$.
\item[(ii)] $D_n \ge D$ for all $n \in \mathbb{N}$.
\end{itemize}

Observe that in case (i), applying \eqref{eq-Dn-decr}, $D_{n+1} \le D_n$ for all $n \le n_0 - 1$. Then it is enough to show that there exists $n \le \omega$ such that $D_n < D+\eps$. Indeed, if $k \in [\omega, n_0-1]$, then $D_k \le D_\omega \le D_n < D+ \eps$. Otherwise, if $k \ge n_0$, $D_k \le D < D+\eps$. 

For case (ii), as $(D_n)$ is nonincreasing, again we only need to show that there exists $n \le \omega$ such that $D_n < D+\eps$. Consequently, in the following we treat both cases at once.

Consider the sequence $(E_n)$ defined by 
\[
E_n=\left\{
\begin{array}{ll}
D_n, & \mbox{if } n \le \omega,\\
D, & \mbox{otherwise}.
\end{array}
\right.
\]
This is a nonincreasing sequence in $[0,b]$ and we can apply Lemma \ref{lemma-mon-seq} taking $\tau = \Theta^N(\alpha)$ and the function $g$ constantly equal to $N$. Thus, there exists $I \le N\left\lceil \frac{b}{\Theta^N(\alpha)}\right\rceil$ such that $D\le D_{I+n+1}\le D_{I+n}$ for all $n < N$ and 
\begin{equation} \label{thm-main-eq1}
|D_{I+n} - D_{I+m}| \le \Theta^N(\alpha),
\end{equation}
for all $n,m\in [0,N]$. Here we use the fact that if $n \in [0,N]$, then $I + n \le \omega$, so $E_{I+n} = D_{I+n}$. Assume that for all $n \in [0, N]$, $D_{I+n} \ge D+\eps$. Denoting $\gamma = D_{I+N} - D \ge \eps$, we have
\begin{align*}
D+\gamma & = D_{I+N} \le D_{I+n} \\
& \le D_{I+N} +\Theta^N(\alpha) \quad \text{by }\eqref{thm-main-eq1}\\
& = D + \gamma + \Theta^N(\alpha),
\end{align*}
hence
\begin{equation} \label{thm-main-eq2}
D+\gamma \le D_{I+n} \le D + \gamma + \Theta^N(\alpha),
\end{equation}
for all $n \in [0,N]$.

Denote now $l_n = L_{I+n}$ and $m_n = M_{I+n}$ for $n\in[0,N]$. Then $d(l_n,m_n) = D_{I+n}$.
\begin{claim} For all $n\in [0,N]$,
\begin{enumerate}
\item[(i)]
$d(l_0,l_n)+d(l_n,m_n) \le d(l_0,m_n)+\Theta^{N-n}(\alpha)$,
\item[(ii)] 
$d(l_0,l_n)\ge n(D-\alpha)$.
\end{enumerate}
\end{claim}
\begin{proof}[Proof of Claim]
We use induction. For $n=0$, the two inequalities are obviously true. As our induction hypothesis (I.H.), suppose that (i) and (ii) hold for $n=k \le N-1$. We prove that they also hold for $n=k+1$. We have 
\begin{align*}
d(l_0,l_{k+1})+d(l_{k+1},m_k) & \ge d(l_0,m_k)\\ 
& \ge d(l_0,l_k) +d(l_k,m_k)-\Theta^{N-k}(\alpha) \quad \text{by (i)-I.H.}\\
& = d(l_0,l_k)+d(l_k,l_{k+1})+d(l_{k+1},m_k)-\Theta^{N-k}(\alpha). 
\end{align*}
Hence,
\begin{align*} 
d(l_0,l_{k+1}) & \ge d(l_0,l_k)+d(l_k,l_{k+1})-\Theta^{N-k}(\alpha) \\
& \ge k(D-\alpha)+D-\Theta^{N-k}(\alpha) \quad \text{by (ii)-I.H.}\\
& \ge (k+1)(D-\alpha),
\end{align*}
which proves (ii) for $k+1$. Also, 
\begin{equation}\label{thm-main-claim-eq4}
d(l_0,m_k) \ge d(l_0,l_{k+1}) + d(l_{k+1},m_k)-\Theta^{N-k}(\alpha).
\end{equation}
Next we obtain  
\begin{align*}
d(l_{k+1},m_k)+d(m_k,m_{k+1}) &\le d(l_k,m_k)-D+D = D_{I+k} \le D_{I+N} +\Theta^N(\alpha) \quad \text{by \eqref{thm-main-eq1}}\\
& \le d(l_{k+1},m_{k+1})+\Theta^N(\alpha) \le d(l_{k+1},m_{k+1})+\Theta^{N-k}(\alpha),  
\end{align*}
i.e.
\begin{equation}\label{thm-main-claim-eq7}
d(l_{k+1},m_k)+d(m_k,m_{k+1})\le d(l_{k+1},m_{k+1})+\Theta^{N-k}(\alpha).
\end{equation}
Under the assumption that $\sep\{l_0,l_{k+1},m_k,m_{k+1}\} \ge \alpha$,
relations \eqref{thm-main-claim-eq4} and \eqref{thm-main-claim-eq7} 
imply, using that $\Theta$ is a modulus of uniform betweenness (applied to 
$x:=m_{k+1}$, $y:=m_k$, $z:=l_{k+1}$, $w:=l_0$), that
\[ d(l_0,m_{k+1})\ge d(l_0,l_{k+1})+d(l_{k+1},m_{k+1})-\Theta^{N-k-1}(\alpha),
\] which is (i) for $k+1.$ \\ 
So in order to finish the proof of the claim, 
it remains to verify that $\sep\{l_0,l_{k+1},m_k,m_{k+1}\} \ge \alpha:$
\begin{enumerate}
\item $d(l_0,l_{k+1}) \ge (k+1)(D-\alpha)\ge (k+1)\alpha\ge \alpha$ using 
that we have already proved (ii) for $k+1$ and \eqref{thm-main-eq-eps}.
\item 
$\begin{aligned}[t]d(l_0,m_k) & \ge d(l_0,l_k)+d(l_k,m_k)-\Theta^{N-k}(\alpha) \quad \text{by (i)-I.H.}\\[0.5mm]
& \ge d(l_0,l_k)+d(l_k,m_k)-\alpha \\[0.5mm] 
& \ge k(D-\alpha)+d(l_k,m_k)-\alpha \quad \text{by (ii)-I.H.} \\[0.5mm]
& \ge k(D-\alpha)+D+\gamma-\alpha \\[0.5mm]
& = (k+1)(D-\alpha)+\gamma\ge 3\alpha \quad \text{since, by \eqref{thm-main-eq-eps}, $D\ge 2\alpha$ and $\gamma\ge\eps \ge 2\alpha$}.
\end{aligned}$ 
\item $\begin{aligned}[t]d(l_0,m_{k+1}) & \ge d(l_0,m_k)-d(m_k,m_{k+1}) \ge (k+1)
(D-\alpha)+\gamma-D \\[0.5mm]
& \ge k(D-\alpha)-\alpha+\gamma\ge \alpha.
\end{aligned}$ 
\item $d(l_{k+1},m_k) =D_{I+k}-D\ge \gamma\ge 2\alpha.$ 
\item $d(l_{k+1},m_{k+1}) =D_{I+k+1}\ge D+\gamma\ge 4\alpha.$
\item $\begin{aligned}[t]
d(m_k,m_{k+1}) & \ge d(l_{k+1},m_{k+1})-d(l_{k+1},m_k) \\[0.5mm]
& \ge D+\gamma-(D_{I+k}-D)\ge D+\gamma-\gamma-\Theta^N(\alpha) \quad \text{by \eqref{thm-main-eq2}} \\[0.5mm] 
& \ge D-\alpha\ge \alpha. 
\end{aligned}$  
\end{enumerate}
This ends the proof of the claim. 
\end{proof}
Consequently, 
\begin{align*}
d(l_0,l_N) & \ge N(D-\alpha) \\
& > b+1-N\alpha \quad \text{by }\eqref{eq-diam-A}\\
& \ge b \quad \text{by }\eqref{thm-main-eq-eps},
\end{align*}
a contradiction to the fact that $b \ge \diam(X)$. This shows that there exists $n \le I + N \le \omega$ such that $D_n < D + \eps$. 
\end{proof}

\begin{re}
Instead of assuming $X$ to be bounded it suffices to assume that $(M_n)$ 
is bounded. Indeed, let $B\ge d(M_0,M_n)$ for all $n\in\N.$ Then 
\[ d(M_0,L_n)\le d(M_0,M_n)+d(L_n,M_n)\le B+\max\{ d(L_0,M_0),D\} \] 
and so for all $m,n\in\N$
\[ d(L_n,M_m)\le d(L_n,M_0)+d(M_0,M_m)\le 2B+\max\{ d(L_0,M_0),D\}=:b.\]
Hence we can take this $b$ throughout the proof of the main theorem. In fact,
it suffices to have $d(M_0,M_n)\le B$ for all $n\le\left\lceil \frac{b+1}{D}\right\rceil +1$.
\end{re}

\section{Uniform betweenness in uniquely geodesic spaces}\label{section-UBTW-uniquely-geod} \label{section-UBTW-uniquely-geod}
The goal of this section is to study the relation of the uniform betweenness property with other convexity properties in uniquely geodesic spaces. The main tool will be the concept of uniform uniqueness of geodesics which we define in Subsection \ref{subsection-uniform-uniquely}. We start with an introductory part.

\subsection{Basic notions and concepts}
This section discusses several geometric properties of geodesic metric spaces with emphasis on convexity notions that play an essential role in the study of the uniform betweenness property. We start with a brief account of some basic definitions on geodesic spaces and refer to \cite{Bri99} for a more detailed treatment.

Let $(X,d)$ be a metric space and $x,y \in X$. A {\it geodesic} joining $x$ to $y$ is a mapping $\gamma:[0,l] \subseteq \R \to X$ such that $\gamma(0)=x$, $\gamma(l)=y$ and 
\[d(\gamma(s),\gamma(s'))=|s-s'| \quad \text{for all }s,s'\in[0,l].\] 
It follows that $l=d(x,y)$. We say that a geodesic $\gamma$ {\it starts} at $x$ if $\gamma(0) = x$. If every two points in $X$ are joined by a (unique) geodesic, then $X$ is called a {\it (uniquely) geodesic space}. The image $\gamma([0,l])$ of a geodesic $\gamma$ is called a {\it geodesic segment} with endpoints $x$ and $y$. Suppose $X$ is a geodesic space. A point $z\in X$ belongs to a geodesic segment with endpoints $x$ and $y$ if and only if there exists $t\in [0,1]$ such that 
\[d(z,x)= td(x,y) \quad \text{and} \quad d(z,y)=(1-t)d(x,y).\] 
In this case, if $\gamma$ is the geodesic in question, then $z=\gamma(tl)$. When $t=1/2$, we call such a point $z$ a {\it midpoint} of $x$ and $y$ and also denote it by $m(x,y)$. Two given points $x$ and $y$ in $X$ may be joined by more than one geodesic and thus may have more than one midpoint. If there is a unique geodesic segment with endpoints $x$ and $y$, we denote it by $[x,y]$ and in this case for all $t\in [0,1]$ there exists only one point $z \in X$, denoted by $(1-t)x + ty$, satisfying $d(z,x)= td(x,y)$ and $d(z,y)=(1-t)d(x,y)$. In particular, $x$ and $y$ have a unique midpoint $m(x,y) = (1/2)x+(1/2)y$.
\begin{de}\label{def-str-conv}
Let $(X,d)$ be a geodesic space. We say that $X$ is {\it strictly convex} if for all $z,x,y \in X$ with $x \ne y$ and all midpoints $m(x,y)$ of $x$ and $y$ we have
\[d(z, m(x,y)) < \max\{d(z,x),d(z,y)\}.\]
\end{de}
Strictly convex geodesic spaces are uniquely geodesic. Indeed, let $\gamma_1$ and $\gamma_2$ be two geodesics joining $x$ to $y$. Denote $u_s=\gamma_1(s)$ and $v_s=\gamma_2(s)$, where $s \in [0,d(x,y)]$. If $u_s \ne v_s$, then taking any midpoint $m(u_s,v_s)$ it follows that
\begin{align*}
d(x,y) & \le d(x,m(u_s,v_s)) + d(y,m(u_s,v_s)) \\
& < \max\{d(x,u_s),d(x,v_s)\} + \max\{d(y,u_s),d(y,v_s)\} = d(x,y),
\end{align*}
a contradiction. Hence $u_s=v_s$ for any $s \in [0,d(x,y)]$, which shows that $\gamma_1=\gamma_2$. This also shows that in Definition \ref{def-str-conv} one can equivalently consider some midpoint of $x$ and $y$ instead of all midpoints as this is enough to prove the uniqueness of geodesics.

Any normed vector space is a geodesic space. For this class of spaces, strict convexity is actually equivalent to the existence of unique geodesics between any two points. However, in general this equivalence fails to hold as the following example shows.
\begin{ex} \label{ex-sphere}
The {\it $2$-dimensional sphere $\mathbb{S}^2$} is the set $\left\{u \in {\R}^{3} : (u \mid u) = 1\right\}$, where $(\cdot\mid \cdot)$ is the Euclidean scalar product. Endowed with the distance $d: \mathbb{S}^2 \times \mathbb{S}^2 \to \R$ that assigns to each $(x,y) \in \mathbb{S}^2 \times \mathbb{S}^2$ the unique number $d(x,y) \in [0,\pi]$ such that $\cos d(x,y) = (x\mid y)$, $\mathbb{S}^2$ is a geodesic space called the spherical space. Any octant of $\mathbb{S}^2$ is a uniquely geodesic space that is not strictly convex.
\end{ex}

Uniform convexity is a strengthening of strict convexity and was first introduced in the linear case in \cite{Cla36} and in a nonlinear setting in \cite{GoeSekSta80,GoeRei84,Leu07}. Since then it was used in various forms in metric spaces and we consider here the following variant from \cite{Leu07}. 

\begin{de} \label{def-UC}
A geodesic space $(X,d)$ is {\it uniformly convex} if for all $\eps \in (0,2]$ and $r>0$ there exists $\delta \in (0,1]$ such that for all $z, x, y\in X$ and all midpoints $m(x,y)$
we have 
\begin{eqnarray}\nonumber
\left.\begin{array}{l}
d(z,x)\le r\\
d(z,y)\le r\\
d(x,y)\ge\varepsilon r
\end{array}
\right\}
& \Rightarrow & d(z,m(x,y))\le (1-\delta)r.
\end{eqnarray}
A mapping $\eta : (0,2] \times (0,\infty) \to (0,1]$ providing for given $r > 0$ and $\eps \in (0, 2]$ such a $\delta = \eta(\eps,r)$ is called {\it a modulus of uniform convexity}. A modulus of uniform convexity is said to be {\it monotone} if it is nonincreasing in the second argument.
\end{de}

Every uniformly convex geodesic space is strictly convex, hence uniquely geodesic. Again, uniform convexity is in fact equivalent to the condition obtained by considering the above implication for some midpoint of $x$ and $y$ instead of all midpoints. Besides, one can show that every compact strictly convex geodesic space is uniformly convex.

In normed vector spaces that are uniformly convex in the sense of Definition \ref{def-UC}, by rescaling balls, one can always find moduli of uniform convexity that do not depend on the second argument, namely on the radii $r$. In fact, one usually considers the notion of {\it the modulus of convexity} of a normed vector space $X$ defined as the function $\delta : [0,2] \to [0,1]$ given by
\[\delta(\eps) = \inf \left\{ 1 - \left\|\frac{x+y}{2}\right\| : \|x\| \le 1, \|y\| \le 1, \|x-y\| \ge \eps\right\},\]
or equivalently,
\[\delta(\eps) = \inf \left\{ 1 - \left\|\frac{x+y}{2}\right\| : \|x\| = 1, \|y\| = 1, \|x-y\| \ge \eps\right\}.\]
Note that $\delta$ is nondecreasing on $[0,2]$ and continuous on $[0,2)$. The normed vector space $X$ is uniformly convex (in the sense of Definition \ref{def-UC} and equivalently in the sense of \cite{Cla36}) if and only if $\delta(\eps) > 0$ for all $\eps > 0$. In this case $\delta$ is the largest possible modulus of uniform convexity one can define for $X$. 

For $1 < p < \infty$, an $L_p$ space over a measurable space is uniformly convex and, if $\delta$ is its modulus of convexity and $\eta : (0,2] \to (0,1]$ is defined by
\begin{equation}\label{eq-mod-Lp}
\eta(\eps)=\left\{
\begin{array}{ll}
\displaystyle\frac{p-1}{8}\eps^2, & \mbox{if } 1 < p \le 2,\\
\displaystyle\frac{1}{p2^p}\eps^p, & \mbox{if } 2 < p < \infty,
\end{array}
\right.
\end{equation}
then $\delta(\eps) \ge \eta(\eps)$ for all $\eps \in (0,2]$. Hence, $\eta$ is a modulus of uniform convexity for $L_p$.

A related notion is {\it the characteristic of convexity} of a normed vector space defined as the number
\[\eps_0 = \sup\{\eps \in [0,2]: \delta(\eps) = 0\}.\]
Then $X$ is uniformly convex if and only if $\eps_0 = 0$. In addition, $\delta$ is strictly increasing on $[\eps_0,2]$. These concepts and proofs of the properties mentioned above can be found e.g. in \cite[Chapter 5]{GK90}.

A particular notion of uniform convexity, called $p$-uniformly convexity, was introduced by Ball, Carlen and Lieb \cite{BalCarLie94} in the linear case and, more recently, in the setting of geodesic spaces by Naor and Silberman \cite{NS11} in the following way: given $1<p<\infty$, a geodesic space $(X,d)$ is {\it $p$-uniformly convex} if there exists a parameter $c>0$ such that for all $x,y,z\in X$, all $t\in[0,1]$ and all geodesics $\gamma$ joining $x$ to $y$,
\begin{equation}\label{eq-p-UC}
d(z,\gamma(td(x,y)))^p\leq(1-t)\,d(z,x)^p+t\,d(z,y)^p-\frac{c}{2}\,t(1-t)\,d(x,y)^p.
\end{equation}
Thus, a geodesic space that is $p$-uniformly convex with parameter $c>0$ is uniformly convex (in the sense of Definition \ref{def-UC}) and admits a modulus of uniform convexity that does not depend on the second argument
\begin{equation} \label{eq-mod-p-UC}
\eta(\eps) = \frac{c}{8p} \eps^p.
\end{equation}
Estimations on $c$ depending on the value of $p$ were give in \cite{Kuw14}. Namely, $c \le 2(p-1)$ if $p \in (1,2)$ and $c \le 8/2^p$ if $p \in [2,\infty)$.

Every $L_p$ space over a measurable space is $p$-uniformly convex if $p > 2$ and $2$-uniformly convex if $p \in (1,2]$. As for geodesic spaces, every $\CAT(0)$ space is $2$-uniformly convex with parameter $c=2$ and, in this case, (\ref{eq-p-UC}) provides a characterization of $\CAT(0)$ spaces. For $\kappa>0$, any $\CAT(\kappa)$ space $X$ with $\diam(X)<\pi/(2\sqrt{\kappa})$ is $2$-uniformly convex with parameter $c=(\pi-2\sqrt{\kappa}\,\eps)\,\tan(\sqrt{\kappa}\,\eps)$ for any $0<\eps\leq\pi/(2\sqrt{\kappa})-\diam(X)$, see \cite{Oht07}. We remark at this point that $\CAT(\kappa)$ spaces are defined in terms of comparisons with the model planes i.e. the complete simply connected $2$-dimensional Riemannian manifolds of constant sectional curvature $\kappa$. More precisely, in $\CAT(\kappa)$ spaces, geodesic triangles (which consist of three points and three geodesic segments joining them) are `thin' when compared to triangles with the same side lengths in the model planes. Note also that a normed real vector space which is $\CAT(\kappa)$ for some $\kappa \in \R$ is pre-Hilbert. A comprehensive exposition of $\CAT(\kappa)$ spaces can be found in \cite{Bri99}.
\subsection{Uniform uniquely geodesic spaces} \label{subsection-uniform-uniquely}
Let $(X,d)$ be a geodesic space. As we pointed out in the introduction, we need a quantitative uniform version of the property that there exists exactly one geodesic joining two points in $X$, and we define it next. Note that every geodesic space that satisfies the condition from below is uniquely geodesic. 

\begin{de}\label{def-UU}
We say that $X$ is {\it uniformly uniquely geodesic} if for all $\eps, b > 0$ there exists $\varphi > 0$ such that for all $x, y, z, w\in X$ with $d(x,y) \le b$ and all $t\in [0,1]$ we have
\begin{eqnarray}\nonumber
\left.\begin{array}{l}
\max\{d(z,x), d(w,x)\} \le td(x,y)\\
\max\{d(z,y), d(w,y)\} \le (1-t)d(x,y) + \varphi\\
\end{array}
\right\}
& \Rightarrow & d(z,w) \le \eps.
\end{eqnarray}
A mapping $\Phi : (0,\infty)^2\to (0,\infty)$ providing for given $\eps, b > 0$ such a $\varphi = \Phi(\eps,b)$ is called {\it a modulus of uniform uniqueness}.
\end{de}
A somehow related property for CAT$(0)$ spaces can be found in \cite[Lemma 17.20]{K08} and \cite[Chapter II, Lemma 9.15]{Bri99}.

\begin{pro} \label{prop-compact-uniq-geod}
Compact uniquely geodesic spaces are uniformly uniquely geodesic. 
\end{pro}
\begin{proof}
We argue by contradiction. Suppose that $(X,d)$ is compact and uniquely geodesic, but not uniformly uniquely geodesic. Then there exists $\eps > 0$ such that for all $n \in \mathbb{N}$ we can find points $x_n, y_n, z_n, w_n \in X$ and numbers $t_n \in [0,1]$ satisfying 
\[\max\{d(z_n,x_n), d(w_n,x_n)\} \le t_nd(x_n,y_n), \quad \max\{d(z_n,y_n), d(w_n,y_n)\} \le (1-t_n)d(x_n,y_n) + \frac{1}{n}\]
and 
\begin{equation} \label{eq1-cmp-uu}
d(z_n, w_n) > \eps. 
\end{equation}
By compactness, we may assume that there exist $x,y,z,w \in X$ and $t \in [0,1]$ such that $x_n \to x$, $y_n \to y$, $z_n \to z$, $w_n \to w$ and $t_n \to t$. Then
\[d(x,y) \le d(z,x) + d(z,y) \le td(x,y) + (1-t)d(x,y) = d(x,y),\]
from where $z \in [x,y]$ with $d(z,x) = td(x,y)$. In the same way, $w \in [x,y]$ with $d(w,x) = td(x,y)$, which shows that $z=w$. This contradicts \eqref{eq1-cmp-uu}.
\end{proof}

In normed vector spaces that are uniformly uniquely geodesic, by rescaling balls, it is enough to define $\Phi(\cdot, 1)$ in order to obtain a modulus of uniform uniqueness: one can take $\Phi(\eps,b) = b\Phi(\eps/b,1)$ for all $\eps, b > 0$. 

We show next that uniform convexity with a monotone modulus of uniform convexity $\eta$ implies uniform uniqueness of geodesics and one can define a modulus of uniform uniqueness in terms of $\eta$.

\begin{thm} \label{thm-UC-UU}
Let $(X,d)$ be a uniformly convex geodesic space that admits a monotone modulus of uniform convexity $\eta$. Then $X$ is uniformly uniquely geodesic and $\Phi : (0,\infty)^2 \to (0,\infty)$ defined by
\[\Phi(\eps,b) = \frac{\eps}{2} \eta\left(\frac{\eps}{b+\eps/2}, b+\eps/2\right)\] 
is a modulus of uniform uniqueness for $X$. 

In addition, if $\eta$ can be written as $\eta(\eps,r) = \eps\tilde{\eta}(\eps,r)$, where $\tilde{\eta}: (0,2] \times (0,\infty) \to (0,1]$ is nondecreasing in $\eps$, then one can take
\[\Phi(\eps,b) = \eps \tilde{\eta}\left(\frac{\eps}{b+\eps}, b+\eps\right).\] 
\end{thm}
\begin{proof}
Let $\eps,b > 0$ and denote 
\[\varphi = \frac{\eps}{2} \eta\left(\frac{\eps}{b+\eps/2}, b+\eps/2\right) \le \frac{\eps}{2}.\] 

Take $x,y,z,w\in X$ with $d(x,y) \le b$ and $t \in [0,1]$ satisfying 
\[\max\{d(z,x), d(w,x)\} \le td(x,y) \quad \text{and} \quad \max\{d(z,y), d(w,y)\} \le (1-t)d(x,y) + \varphi.\]
We need to show that $d(z,w) \le \eps$. Suppose, on the contrary, that $d(z,w) > \eps$. In this case, denoting $r_1 = td(x,y)$ and $r_2 = (1-t)d(x,y)$, we have
\[\eps < d(z,w) \le d(z,y) + d(w,y) \le 2(r_2 + \varphi),\] 
so 
\begin{equation} \label{thm-UC-UU-eq1}
r_2 + \varphi \ge \frac{\eps}{2}.
\end{equation} 
We may assume that $r_1 > 0$. As $\max\{d(z,x), d(w,x)\} \le r_1$ and
\[d(z,w) > \eps = \frac{\eps}{r_1}r_1,\] 
by uniform convexity, 
\begin{equation} \label{thm-UC-UU-eq2}
d\left(x,m(z,w)\right)\le \left(1-\eta\left(\frac{\varepsilon}{r_1},r_1\right)\right)r_1.
\end{equation}
At the same time, since $\max\{d(z,y), d(w,y)\} \le r_2 + \varphi$ and
\[d(z,w) > \eps \ge \frac{\eps}{b+\eps/2}(r_2+\varphi),\] 
by uniform convexity, 
\begin{equation} \label{thm-UC-UU-eq3}
d\left(y,m(z,w)\right)\le \left(1-\eta\left(\frac{\varepsilon}{b+\eps/2},r_2 + \varphi\right)\right)(r_2 + \varphi).
\end{equation}
Hence,
\begin{align*}
r_1 + r_2 & = d(x,y) \le d(x,m(z,w)) + d(y,m(z,w))\\
& \le \left(1-\eta\left(\frac{\eps}{r_1},r_1\right)\right)r_1 + \left(1-\eta\left(\frac{\eps}{b+\eps/2},r_2 + \varphi\right)\right)(r_2 + \varphi) \quad \text{by } \eqref{thm-UC-UU-eq2} \text{ and }\eqref{thm-UC-UU-eq3}\\
& \le r_1 + r_2 - r_1\eta\left(\frac{\eps}{r_1},r_1\right) + \varphi - \frac{\eps}{2}\eta\left(\frac{\eps}{b+\eps/2},r_2 + \varphi\right) \quad \text{by }\eqref{thm-UC-UU-eq1}.
\end{align*} 
Using the monotonicity of $\eta$ we obtain
\[0 < r_1\eta\left(\frac{\varepsilon}{r_1 },r_1\right)\le \varphi - \frac{\eps}{2} \eta\left(\frac{\eps}{b+\eps/2}, b+\eps/2\right),\]
a contradiction.

Suppose now $\eta(\eps,r) = \eps\tilde{\eta}(\eps,r)$ with $\tilde{\eta}$ nondecreasing in $\eps$. For $\eps,b > 0$, let
\[\varphi = \eps \tilde{\eta}\left(\frac{\eps}{b+\eps}, b+\eps\right) \le \eps.\] 
As $\max\{d(z,y), d(w,y)\} \le r_2 + \varphi$ and 
\[d(z,w) > \eps = \frac{\eps}{r_2+\varphi}(r_2+\varphi),\] 
by uniform convexity and the monotonicity of $\eta$ we have
\begin{align*}
d\left(y,m(z,w)\right)&\le \left(1-\eta\left(\frac{\varepsilon}{r_2 + \varphi},r_2 + \varphi\right)\right)(r_2 + \varphi)\\
& \le \left(1-\eta\left(\frac{\varepsilon}{r_2 + \varphi},b+\eps\right)\right)(r_2 + \varphi)\\
& = \left(1-\frac{\varepsilon}{r_2 + \varphi}\tilde{\eta}\left(\frac{\varepsilon}{r_2 + \varphi},b+\eps\right)\right)(r_2+\varphi).
\end{align*}
Using the monotonicity of $\tilde{\eta}$ we obtain
\begin{equation} \label{thm-UC-UU-eq4}
d\left(y,m(z,w)\right) \le \left(1-\frac{\varepsilon}{r_2 + \varphi}\tilde{\eta}\left(\frac{\varepsilon}{b+\eps},b+\eps\right)\right)(r_2+\varphi). 
\end{equation}
The same reasoning as before applying now \eqref{thm-UC-UU-eq4} instead of \eqref{thm-UC-UU-eq3} finishes the proof.
\end{proof}

In particular, using \eqref{eq-mod-Lp}, for $1 < p < \infty$, $L_p$ spaces over measurable spaces admit a modulus of uniform uniqueness $\Phi : (0,\infty)^2\to (0,\infty)$ defined by
\begin{equation} \label{eq-mod-uniq-Lp}
\Phi(\eps,b)=\left\{
\begin{array}{ll}
\displaystyle\frac{p-1}{8}\frac{\eps^2}{(b+\eps)}, & \mbox{if } 1 < p \le 2,\\
\displaystyle\frac{1}{p2^p}\frac{\eps^{p}}{(b+\eps)^{p-1}}, & \mbox{if } 2 < p < \infty.
\end{array}
\right.
\end{equation}
If $X$ is $p$-uniformly convex with parameter $c$, then, according to \eqref{eq-mod-p-UC},
\begin{equation} \label{eq-mod-uniq-p-UC}
\Phi(\eps,b) = \frac{c}{8p}\frac{\eps^p}{(b+\eps)^{p-1}}, 
\end{equation}
acts as a modulus of uniform uniqueness for $X$. 

Revisiting Example \ref{ex-sphere} we can immediately notice that there exist uniformly uniquely geodesic spaces that are not uniformly convex. Indeed, any octant of ${\mathbb{S}}^2$ is a compact uniquely geodesic space, thus, by Proposition \ref{prop-compact-uniq-geod}, it is uniformly uniquely geodesic. However, it is not strictly convex, and hence not uniformly convex. 

On the other hand, recall that in normed vector spaces, strict convexity is equivalent to uniqueness of geodesics. This equivalence still holds when passing to the uniform versions of these properties. Namely, Theorem \ref{thm-UC-UU} and 
Theorem \ref{thm-normed-UU-UC}  below show that in normed vector spaces 
uniform uniqueness of 
geodesics is equivalent to uniform convexity, and respective moduli can be expressed in terms of each other. Before proving Theorem \ref{thm-normed-UU-UC}, we recall the following property of the modulus of convexity. Its proof can be found e.g. in \cite[p. 56]{GK90}, but since it is short, for completeness we include it below.

\begin{lem}[Goebel and Kirk \cite{GK90}] \label{lemma-mod-conv}
Let $(X,\|\cdot\|)$ be a normed vector space with modulus of convexity $\delta$ and characteristic of convexity $\eps_0$. If $\eps_0 < 2$, then 
\[\delta(2(1-\delta(\eps))) \le 1 - \frac{\eps}{2},\]
for all $\eps \in (\eps_0,2]$. 
\end{lem}
\begin{proof}
Let $\eps \in (\eps_0,2]$. Clearly, if $\delta(\eps) = 1$ (which can only happen for $\eps = 2$), then the inequality holds. Moreover, $\delta(\eps) > 0$ and so we can assume that $\delta(\eps) \in (0,1)$. Let $\tau \in (0,1-\delta(\eps))$ and take $x, y \in X$ with 
\[\|x\|=\|y\|=1, \quad \|x-y\|\ge \eps, \quad \text{and} \quad \left\|\frac{x+y}{2}\right\| \ge 1 - \delta(\eps) - \tau.\] 
As $\delta$ is nondecreasing, we get $\delta(\|x+y\|) \ge \delta(2(1-\delta(\eps)-\tau))$. Furthermore,
\[\frac{\eps}{2} \le \frac{\|x-y\|}{2} = \frac{\|x+(-y)\|}{2} \le 1 - \delta(\|x-(-y)\|) = 1 - \delta(\|x+y\|).\]
Hence, 
\[\delta(2(1-\delta(\eps)-\tau)) \le 1 - \frac{\eps}{2}\]
and we only need to let $\tau \searrow 0$ to obtain the desired inequality.
\end{proof}

\begin{thm}\label{thm-normed-UU-UC}
Let $(X,\|\cdot\|)$ be a normed vector space that is uniformly uniquely geodesic with a modulus of uniform uniqueness $\Phi$ satisfying $\Phi < 1$. Then $X$ is uniformly convex and its modulus of convexity $\delta$ can be estimated by 
\[\delta(\eps) \ge \frac{1}{2}\Phi\left(\frac{\eps}{3},1\right),\] 
for all $\eps \in (0,2]$. In particular, $\eta:(0,2] \rightarrow (0,1]$ defined by 
\[\eta(\eps) = \frac{1}{2}\Phi\left(\frac{\eps}{3},1\right)\] 
is a modulus of uniform convexity for $X$.
\end{thm}
\begin{proof}
Let $\eps_0$ be the characteristic of convexity of $X$ and denote for simplicity $\varphi : (0,\infty) \to (0,1)$, $\varphi(\eps)=\Phi(\eps,1)$. 

We show first that $\eps_0 < 2$. If $\eps_0 = 2$, then $\delta(2-\varphi(1/2)) = 0$, so there exist $x,y \in X$ with 
\[\|x\|=\|y\|=1, \quad \|x-y\| \ge 2-\varphi(1/2), \quad \text{and} \quad \left\|\frac{x+y}{2}\right\| > \frac{1}{2}.\] 
Denote $t=\left(2-\varphi(1/2)\right)^{-1}$. Then $1-t = \left(1 - \varphi(1/2)\right)\left(2-\varphi(1/2)\right)^{-1}$,
\[\|x\| = 1 \le t\|x-y\| \quad \text{and} \quad \|y\| = 1 \le (1-t)\|x-y\| + \varphi(1/2).\]
Letting $w = (1-t)x+ty$, by uniform uniqueness, 
\[\left\|\frac{x+y}{2}\right\| \le \|w\| \le \frac{1}{2},\] 
a contradiction. 

Since $\delta$ is continuous on $[0,2)$, we have $\delta(\eps_0) = 0$. Suppose that $\eps_0 > 0$. Then we can take $\eps \in (\eps_0,2)$ such that $\delta(\eps) < \varphi(\eps_0/2)/2$. Applying Lemma \ref{lemma-mod-conv}, there exist $x,y \in X$ such that 
\[\|x\|=\|y\|=1, \quad \|x-y\|\ge 2(1-\delta(\eps)), \quad \text{and} \quad \left\|\frac{x+y}{2}\right\| > \frac{\eps_0}{2}.\] 
Let now $t=\left(2-\varphi(\eps_0/2)\right)^{-1}$. Then $1-t = \left(1 - \varphi(\eps_0/2)\right)\left(2-\varphi(\eps_0/2)\right)^{-1}$,
\[\|x\| \le t\|x-y\| \quad \text{and} \quad \|y\| \le (1-t)\|x-y\| + \varphi(\eps_0/2).\]
Denoting again $w = (1-t)x+ty$, by uniform uniqueness, 
\[\left\|\frac{x+y}{2}\right\| \le \|w\| \le \frac{\eps_0}{2},\] 
another contradiction. Therefore, $\eps_0 = 0$, so $X$ is uniformly convex.

Assume now that $\delta(\eps) < \varphi(\eps/3)/2$ for some $\eps \in (0,2]$. By Lemma \ref{lemma-mod-conv}, there exist $x,y \in X$ satisfying
\[\|x\|=\|y\|=1, \quad \|x-y\| \ge 2(1-\delta(\eps)), \quad \text{and} \quad \left\|\frac{x+y}{2}\right\| > \frac{\eps}{3}.\] 
Arguing as before one shows that the assumption is false, hence the desired estimate holds.
\end{proof}
\subsection{Uniform uniquely geodesic spaces and uniform betweenness}
In a uniquely geodesic space $X$, the betweenness property $\BW$ given in Definition \ref{definition-betweenness} can be reformulated in the following equivalent way: for every four pairwise distinct points $x,y,z,w \in X$, 
\begin{eqnarray}\nonumber
\left.\begin{array}{l}
y \in [x,z]\\
z \in [y,w]
\end{array}
\right\}
& \Rightarrow & y,z \in [x,w].
\end{eqnarray}
Moreover, $\BW$ is equivalent to strong convexity and there exist uniquely geodesic spaces that do not satisfy $\BW$ (see \cite{LNP1} and \cite{Ka}). Recall however that condition \ref{eq-metric-conv} implies $\BW$. We now show that uniformly uniquely geodesic spaces where \eqref{eq-metric-conv} holds satisfy $\UBW$ (see Definition \ref{def-uniform-btw}). In addition, given a modulus of uniform uniqueness, one can convert it into a modulus of uniform betweenness. 

\begin{thm}\label{thm-UU-UBTW}
Let $(X,d)$ be a uniformly uniquely geodesic space with a modulus of uniform uniqueness $\Phi.$ Additionally, suppose that \eqref{eq-metric-conv} holds for all $x,y,z \in X$ and all $t \in [0,1]$. Then $X$ satisfies $\UBW$ and the mapping $\Theta :(0,\infty)^3\ \to (0,\infty )$ defined by
\[\Theta(\eps,a,b) =\min\left\{a, \Phi\left(\min\left\{\frac{a\eps}{8b}, \frac{a}{2}\right\},b\right)\right\}\]
is a modulus of uniform betweenness for $X$.
\end{thm}
\begin{proof}
Let $\eps, a, b > 0$ and $\theta = \Theta(\eps,a,b)$. Take $x,y,z,w \in X$ such that $\sep\{x,y,z,w\}\ge a$, $\diam\{x,y,z,w\} \le b$,
\[d(x,y)+d(y,z)\le d(x,z)+\theta \quad \text{and} \quad d(y,z)+d(z,w)\le d(y,w)+ \theta.\]
As $\theta \le a$ and $d(y,z) \ge a$, there exists $t \in [0,1]$ such that $d(x,y) =td(x,z)$. Similary, one can find $s \in [0,1]$ such that $d(y,z) = sd(y,w)$. Then
\[d(y,z) \le (1-t)d(x,z) + \theta \quad \text{and} \quad d(z,w)\le (1-s)d(y,w)+ \theta.\]
Denoting $y' = (1-t)x+tz$ and $z'=(1-s)y+sw$, by uniform uniqueness, we obtain 
\[\max\left\{d(y,y'),d(z,z')\right\} \le \min\left\{\frac{a\eps}{8b}, \frac{a}{2}\right\}.\]
Using \eqref{eq-metric-conv} we get
\[d(w,y') \le (1-t)d(w,x) + td(w,z),\]
from where
\begin{align*}
(1-t)d(w,x) & \ge d(w,y') - d(w,z) + (1-t)d(w,z)\\
& \ge d(w,y) - d(y,y') - d(w,z') - d(z,z') + (1-t)d(w,z)\\
& = d(y,z') - \left(d(y,y') + d(z,z')\right) + (1-t)d(w,z)\\
& \ge d(y',z) - 2\left(d(y,y') + d(z,z')\right) + (1-t)d(w,z)\\
& = (1-t)d(x,z) - 2\left(d(y,y') + d(z,z')\right) + (1-t)d(w,z)\\
& \ge (1-t)d(x,z) + (1-t)d(w,z) - \frac{a\eps}{2b}.
\end{align*}
Note that $0 < a/(2b) \le 1-t$ because
\[a \le d(y,z) \le d(y,y') + d(y',z) \le \frac{a}{2} + (1-t)d(x,z) \le \frac{a}{2} + (1-t)b.\]
Therefore, $d(x,z) + d(w,z) \le d(x,w) + \eps$.
\end{proof}
\begin{re} 
As a corollary to Theorem \ref{thm-UU-UBTW} and the transformation formulas in 
Proposition 
\ref{prop-modulus-UBW'} we can replace the $\Theta$ above in the linear 
case by 
\[ \Theta'(\eps,a,b):=a\cdot\min\left\{ \frac{1}{2},\frac{\eps}{4b},
\Phi\left(\min\left\{ \frac{\eps}{192b},\frac{1}{4}\right\},3\right) 
\right\}, \] 
which has - for the $\Phi$'s as constructed in (\ref{eq-mod-uniq-Lp}) and 
(\ref{eq-mod-uniq-p-UC}) - a better dependence 
w.r.t. $a$ than $\Theta$ resulting in a better dependence of 
$\Theta'(\eps,\eps,b)$ 
(used in Theorem \ref{thm-UC-UU}) w.r.t $\eps.$
\end{re}

The next remark is based on the discussion from this section and allows us to obtain a large number of spaces where our main results can be applied.

\begin{re}
It is immediate that Theorem \ref{thm-main} and Corollary \ref{co-main} apply in particular for $L_p$ spaces over measurable spaces with $1 < p < \infty$ and $\CAT(\kappa)$ spaces with diameter smaller than $\pi/(2\sqrt{\kappa})$ if $\kappa > 0$. For these classes of spaces, using \eqref{eq-mod-uniq-Lp}, \eqref{eq-mod-uniq-p-UC} and Theorem \ref{thm-UU-UBTW}, we have an explicit modulus of uniform betweenness $\Theta$ and therefore we can compute the exact expression of the rate of convergence provided by Corollary \ref{co-main}. Disregarding the quantitative aspect, in the setting of $\CAT(\kappa)$ spaces, this recovers corresponding results from \cite{AleBisGhr10}. Observe also that Theorem \ref{thm-main} guarantees the success of the lion when the Lion-Man game is played in a bounded convex subset of a uniformly convex normed space. 
\end{re}

\section{Uniform betweenness without unique geodesics}\label{section-UBTW-nonuniquely-geod}
In this section we consider two particular instances of geodesic metric spaces where geodesics joining two points are not necessarily unique, but which still satisfy $\UBW$: Ptolemy spaces and a certain nonstrictly convex normed space of dimension $3$. In each case we compute a modulus of uniform betweenness.

\subsection{Ptolemy spaces}

\begin{de} A metric space $(X,d)$ is called a {\it Ptolemy space} if 
\[ d(x,z)d(y,w)\le d(x,y)d(z,w)+d(x,w)d(y,z),\]
for all $x,y,z,w\in X$.
\end{de}

In normed spaces, the above inequality (with the natural metric $d(x,y) = \|x-y\|$ for $x,y \in X$) provides a characterization of inner product spaces. In the geodesic setting, Ptolemy spaces proved to be significant in the study of the boundary at infinity of $\CAT(-1)$ spaces (see \cite{FoeSch11}). Every $\CAT(0)$ space is a Ptolemy space, but there exist complete bounded geodesic Ptolemy spaces that are not uniquely geodesic. However, in the presence of Busemann convexity, Ptolemy spaces satisfy the $\CAT(0)$ condition (see \cite{FoeLytSch07}).

In \cite{Nic13}, it was shown that Ptolemy metric spaces have $\BW$. The proof can easily be seen to establish even $\UBW$ and the following modulus can be extracted.
\begin{pro} \label{prop-modulus-Ptolemy}
Let $(X,d)$ be a Ptolemy space. Then $\Theta(\eps,a,b):=
\sqrt{b^2+a\eps}-b$ is a modulus of uniform betwenness.
\end{pro}
\begin{proof} 
Let $\eps, a, b > 0$ and $\theta = \Theta(\eps,a,b)$. Take $x,y,z,w\in X$ with $\sep\{ x,y,z,w\}\ge a>0$ and 
$\diam\{ x,y,z,w\}\le b.$
Let 
\[ d(x,y)+d(y,z)\le d(x,z)+\theta \ \mbox{and} \ 
d(y,z)+d(z,w)\le d(y,w)+\theta.\]
Then 
\begin{align*}
& d(x,y)d(y,z)+d(y,z)^2+d(x,y)d(z,w)+d(y,z)d(z,w) = \left( d(x,y)+d(y,z)\right)\left( d(y,z)+d(z,w)\right) \\ 
\quad & \le (d(x,z)+\theta)(d(y,w)+\theta)\le d(x,z)d(y,w)+2\theta b+\theta^2\\
\quad & \le d(x,y)d(z,w)+d(x,w)d(y,z)+2\theta b+\theta^2.
\end{align*}
Hence 
\[ d(x,y)d(y,z)+d(y,z)^2+d(y,z)d(z,w)\le 
d(x,w)d(y,z)+2\theta b+\theta^2 \] and so dividing  by $d(y,z)\ge a$ 
gives 
\[ d(x,z)+d(z,w)\le d(x,y)+d(y,z)+d(z,w)\le d(x,w)+\frac{2\theta b+\theta^2}{a}
=d(x,w)+\eps.\] 
\end{proof}

\subsection{A renorming of $\mathbb{R}^3$}

In \cite{DW}, the authors considered the normed space $(\R^3,\|\cdot\|)$, where for $(x,y,z) \in \R^3$,
\[ \| (x,y,z)\|:=\sqrt{|z^2-(x^2+y^2)|+3z^2+x^2+y^2},\]
and showed that it is not strictly convex, but satisfies $\BW$. Here we extract a modulus of uniform betweenness by analyzing 
the proof from \cite{DW} quantitatively. In the rest of this section 
$\|\cdot\|$ refers to the above norm, while $\|\cdot\|_2$ denotes the 
Euclidian norm on $\R^3.$ We need several lemmas, the first two of which 
are easy to check.

\begin{lem}\label{l1}  
The function $\sqrt{\cdot}$ is $1$-Lipschitz on $[1/4,\infty)$ and $2$-Lipschitz on $[1/16,\infty).$
\end{lem}

\begin{lem}\label{l2}
For all $u \in \R^3$, $\| u\|_2\le \| u\| \le 2\| u\|_2.$
\end{lem}

Define $X_1:=\{ (x,y,z) \in \R^3: z^2\le x^2+y^2\}$, $X_2:=\{ (x,y,z) \in \R^3: z^2\ge x^2+y^2\}$, and given $\eps > 0$, $X^{\eps}_2:=\{ 
(x,y,z) \in \R^3: z^2+\eps\ge x^2+y^2\}$. Clearly, for $u \in \R^3$,
\[
\|u\|=\left\{
\begin{array}{ll}
\sqrt{2}\| u\|_2, & \mbox{if } u\in X_1,\\
2|z|, & \mbox{if } u\in X_2.
\end{array}
\right.
\]

\begin{lem}\label{l3}
If $0<\eps\le 1/4$, $u=(x,y,z)\in X^{\eps/4}_2$ and $1\ge \| u\|\ge 1-\eps,$ then $1 -  2\eps \le 2|z|\le 1+\eps$.
\end{lem}
\begin{proof} Using Lemma \ref{l2} and the assumptions we have
\begin{equation}\label{l3-eq1}
x^2+y^2+z^2 = \| u\|^2_2 \ge \frac{1}{4}\|u\|^2 \ge \frac{9}{64} \ge \frac{1}{16}.
\end{equation}
Applying now Lemma \ref{l1} and using that 
$u\in X^{\eps/4}_2$ we get
\[\left| \|u\|-\sqrt{4z^2+\frac{\eps}{4}}\right|=
\left| \| u\| -\sqrt{\left| z^2+\frac{\eps}{4}-(x^2+y^2)\right|
+3z^2+x^2+y^2}\right| \le 2\cdot\frac{\eps}{4}=\frac{\eps}{2}.\]
By \eqref{l3-eq1}, we can deduce that $2z^2 + \eps/4 \ge 9/64$, hence, $4z^2\ge 1/16$. Using again Lemma \ref{l1},
\[ \left| \| u\|-2|z|\right|\le\left|\sqrt{4z^2+\frac{\eps}{4}}-
\sqrt{4z^2}\right|+\frac{\eps}{2} \le \eps. \]
Since $1\ge \| u\|\ge 1-\eps,$ this implies $1+\eps\ge 2|z|\ge 1-2\eps.$
\end{proof} 

\begin{lem}\label{l4} If $0<\eps\le1/8, \, \| u\|=\| v\| =1, \,
u,v\in X^{\eps/4}_2$ and $\left\| \frac{u+v}{2}\right\|\ge 1-1/25,$ 
then $\frac{u+v}{2}\in X^{\eps}_2$.
\end{lem}
\begin{proof} 
Suppose $u=(x_1,y_1,z_1), v=(x_2,y_2,z_2)$ and 
denote $(x,y,z) = u+v$. By Lemma \ref{l3}, 
\[\frac{3}{8} \le \frac{1}{2}-\eps \le |z_i|\le \frac{1}{2}+\frac{\eps}{2} \le \frac{9}{16} \quad \text{for }i \in \{1,2\}.\] 
Case 1: $z_1\cdot z_2 < 0$, say $z_1>0>z_2.$ Then $z^2=\left( |z_1|-|z_2|\right)^2\le (3\eps/2)^2 \le 9/256$.  
From Lemma \ref{l2} and $\| (u+v)/2\|\ge 1-1/25$, it easily follows that $x^2+y^2\ge 1/4 $ and so Lemma \ref{l1}
applied to $|(|z^2-(x^2+y^2)|+3z^2+x^2+y^2)-2(x^2+y^2)|\le 4\cdot\frac{9}{256}$ 
gives 
\[ \left|\|u+v\| - \sqrt{2(x^2+y^2)}\right| = \left|\sqrt{ |z^2-(x^2+y^2)|+3z^2+x^2+y^2}-\sqrt{2(x^2+y^2)}\right| 
\le\frac{9}{64},\]
so
\begin{equation}\label{l4-eq1}
 \left\| \frac{u+v}{2}\right\|  \le \frac{1}{\sqrt{2}}
\sqrt{x^2+y^2}+ \frac{9}{128}.
\end{equation}
Using that $u,v\in X^{\eps/4}_2$ and Lemma \ref{l1} we have
\begin{align*}
\sqrt{ x^2+y^2} &\le \sqrt{x^2_1+y^2_1}+\sqrt{x^2_2+y^2_2} \le \sqrt{z^2_1+\frac{\eps}{4}}+\sqrt{z^2_2+\frac{\eps}{4}}\le |z_1|+\frac{\eps}{2}+|z_2|+\frac{\eps}{2}\le 1 + \eps + \eps \le \frac{5}{4}.
\end{align*}
Using this inequality in \eqref{l4-eq1} we obtain $\|(u+v)/2\| <1-1/25$ and so this case cannot occur. \\
Case 2: $z_1 \cdot z_2 > 0$.  Then $z^2 = (|z_1|+|z_2|)^2$. As before, 
\[x^2+y^2 \le \left(|z_1|+\frac{\eps}{2}+|z_2|+\frac{\eps}{2}\right)^2\le (|z_1|+|z_2|)^2+4\eps=z^2+4\eps,\]
where the last inequality follows since $2(|z_1| + |z_2|) + \eps \le 4$. This ends the proof.
\end{proof}

\begin{lem}\label{l5} Let $u,v\in X_1$ with $\| u\|=\| v\| =1$ and $\eta$ 
be a modulus of uniform convexity for $(\R^3,\|\cdot\|_2),$ e.g. 
$\eta(\eps)=\eps^2/8.$ Let $0<\eps\le 1.$ 
\begin{enumerate}
\item 
If $\frac{u+v}{2}\in X_1$ and $\left\|\frac{u+v}{2}\right \|\ge 1-\eta(\sqrt{2}\cdot\eps),$
then $\| u-v\|_2\le \eps.$
\item
If $\frac{u+v}{2}\in X_2$ and $\left\|\frac{u+v}{2}\right \|\ge 1-\frac{\eps}{2},$ then 
$u,v\in X^{\eps}_2.$
\end{enumerate}
\end{lem}
\begin{proof} 1. Let $\frac{u+v}{2}\in X_1$ and $\left\|\frac{u+v}{2}\right \|\ge 1-\eta(\sqrt{2}\cdot\eps)$ and define $\tilde{u}:=\sqrt{2}u$, $\tilde{v}:=
\sqrt{2}v.$ Since $(u+v)/2\in X_1$ we have 
\[ \left\| \frac{\tilde{u}+\tilde{v}}{2}\right\|_2=
\sqrt{2}\left\| \frac{u+v}{2}\right\|_2=\left\| \frac{u+v}{2}\right\|\ge 
1-\eta(\sqrt{2}\cdot\eps). \] Thus
$\sqrt{2}\| u-v\|_2 =\| \tilde{u}-\tilde{v}\|_2\le \sqrt{2}\cdot \eps$ 
since $1=\| u\| =\| \tilde{u}\|_2$ and analogously for $v.$\\\
2. Suppose $u=(x_1,y_1,z_1)$ and $v=(x_2,y_2,z_2)$. Since $\frac{u+v}{2}\in X_2$ and $|z_1+z_2|=\left\| \frac{u+v}{2}\right\| \ge 
1-\eps/2,$ it follows that $|z_1| + |z_2| \ge |z_1+z_2|\ge 1-\eps/2.$
As $u\in X_1,\| u\|=1$ we have $1/2=x^2_1+y^2_1+z^2_1\ge 2z^2_1$ and 
so $|z_1|\le1/2$. Likewise, $|z_2|\le1/2$. Consequently, 
$(1-\eps)/2\le |z_i|\le 1/2$ for $i \in \{1,2\}$ and so 
\[ x^2_1+y^2_1=\frac{1}{2}-z^2_1\le \frac{1}{2}-\frac{(1-\eps)^2}{4}
\le \frac{1}{2}-\frac{1-2\eps}{4}= \frac{1+2\eps}{4}.\]
Since $z^2_1\ge (1-\eps)^2/4\ge (1-2\eps)/4,$
we have that $z^2_1+\eps\ge x^2_1+y^2_1,$ i.e. $u\in X^{\eps}_2$ and, likewise,
$v\in X^{\eps}_2.$  
\end{proof}

\begin{lem}\label{l6}
Let $\eta$ be as in Lemma \ref{l5}, $u\in X_1,v\in X_2, \| u\|=\| v\| =1, 
\, 0<\eps\le 1$ and 
\[ \left\| \frac{u+v}{2}\right\| \ge 1-\min\left\{ \frac{\eta\left(
\frac{\sqrt{2}\cdot\eps}{4}\right)}{\sqrt{2}},
\frac{\eps}{2}\right\}. \]
\begin{enumerate}
\item
If $\frac{u+v}{2}\in X_1,$ then $\| u-v\|\le \eps.$
\item
If $\frac{u+v}{2}\in X_2,$ then $u\in X^{\eps}_2.$
\end{enumerate}
\end{lem}
\begin{proof} 1. Suppose $v=(x_2,y_2,z_2)$. By the assumptions on $u$ and $v$ and Lemma \ref{l2} we know that 
\[\|u\|_2 = 1/\sqrt{2}=\sqrt{2z_2^2} \ge \|v\|_2 \ge 1/2.\]
Let $\tilde{\eps}:=\sqrt{2}\cdot \eps/4.$ Using the assumptions on 
$\frac{u+v}{2}$ we then have 
\[ \sqrt{2}\left( 1-\frac{\eta(\tilde{\eps})}{\sqrt{2}}\right)\le \frac{1}{\sqrt{2}}\|u+v\| =
\| u+v\|_2\le \| u\|_2+\| v\|_2 \le \sqrt{2} \]
and so $\| u\|_2+\| v\|_2 \le \| u+v\|_2 +\eta(\tilde{\eps})$. One can now show that
\[ \left\| \frac{\frac{u}{\| u\|_2}+\frac{v}{\| v\|_2}}{2}\right\|_2\ge 
1-\eta(\tilde{\eps}), \]
which yields, by uniform convexity, $\left\| \frac{u}{\| u\|_2}- \frac{v}{\| v\|_2}\right\|_2\le 
\tilde{\eps}.$ Then, 
\[ \left\| \frac{u\| v\|_2}{\| u\|_2}-v\right\|\le 2\left\| \frac{u\| v\|_2}{\| u\|_2}-v\right\|_2 \le 
2\tilde{\eps} \| v\|_2 \le \frac{\eps}{2}. \]
As $\|u\|=\| v\|=1$ we obtain
\[ \left| \frac{\| v\|_2}{\| u\|_2}-1\right| =
\left| \left\| \frac{u\| v\|_2}{\| u\|_2}\right\|-\| v\|\right|
\le \left\| \frac{u\| v\|_2}{\| u\|_2} -v\right\| \le \frac{\eps}{2}.\]
Put together 
\[ \| u-v\| \le
\left\| \frac{u\| v\|_2}{\| u\|_2}-v\right\| +\left\| \frac{u\| v\|_2}{\| 
u\|_2}-u\right\| \le \frac{\eps}{2}+\left| \frac{\| v\|_2}{\| u\|_2}-1\right|
\cdot\| u\| \le \eps.\]
2. Here the reasoning is exactly as in the proof of Lemma \ref{l5}.2 
except that to show that $|z_2|\le 1/2$ we this time use 
$1=\| v\|=2|z_2|$ since $v\in X_2.$
\end{proof}

\begin{pro}\label{prop-modulus-DW}
$\delta(\eps):=\min\left\{ \frac{\eta\left(\frac{\sqrt{2}\cdot\eps}{256}
\right)}{\sqrt{2}},
\frac{\eps}{128}\right\}$ is a modulus for the property 
$\UBW'$ for $(\R^3,\|\cdot \|)$ when $0<\eps\le 1.$
\end{pro}
\begin{proof} Let $u,v,w\in\R^3$ with $\| u\| =\| v\|= \| w\|=1$ and $\min\left\{\left\| \frac{u+v}{2}\right\|,
\left\| \frac{v+w}{2}\right\| \right\}\ge 1-\delta(\eps).$ By Lemmas \ref{l5} 
and \ref{l6} we have either 
\begin{enumerate}
\item[(i)] $\| u-v\| \le\eps/64$ or $\| v-w\|\le \eps/64$ or
\item[(ii)] $u,v,w\in X^{\eps/64}_2.$
\end{enumerate}
(i) W.l.o.g. we can consider the case $\| u-v\|\le\eps/64$. By assumption 
$\| v+w\|\ge \| v\| +\| w\|-\eps/64$. Applying Lemma \ref{lemma-norm-comb}, $\| 2v+w\|\ge 2\| v\|+\| w\| -2\eps/64=3-\eps/32$.
Hence 
\[ \| u+v+w\|\ge \| 2v+w\| -\| u-v\|\ge 3-\frac{\eps}{32}-\frac{\eps}{64}>3-\eps.\]
(ii) Suppose $u=(x_1,y_1,z_1)$, $v=(x_2,y_2,z_2)$, $w=(x_3,y_3,z_3)$ and denote $(x,y,z) = u+v+w$. Lemma \ref{l4} gives $\frac{u+v}{2},\frac{v+w}{2}\in X^{\eps/16}_2$. Applying Lemma \ref{l3} first to $u,v,w\in X^{\eps/64}_2$ we have 
\[ \frac{1}{2}+\frac{\eps}{32}\ge |z_i|\ge\frac{1}{2}-\frac{\eps}{16} \ge \frac{1}{4}\quad \text{for } i \in \{1,2,3\}\] 
and then to $\frac{u+v}{2},\frac{v+w}{2}\in X^{\eps/16}_2$,
\[ 1+\frac{\eps}{4}\ge |z_1+z_2|,|z_2+z_3|\ge 1-\frac{\eps}{2}.\]
Hence $z_1,z_2,z_3$ have the same sign, say $z_1,z_2,z_3>0.$ Using that $u,v,w\in X^{\eps/64}_2$, $z_i^2 \ge 1/16$ for $i \in \{1,2,3\}$, Lemma \ref{l1} and $2z + 3\eps/32 \le 4$ we obtain 
\begin{align*}
x^2+y^2 &\le \left( \sqrt{x^2_1+y^2_1} +\sqrt{x^2_2+y^2_2}+ \sqrt{x^2_3+y^2_3}\right)^2\le
\left( \sqrt{z^2_1+\frac{\eps}{64}}+ \sqrt{z^2_2+\frac{\eps}{64}} +\sqrt{z^2_3+\frac{\eps}{64}}\right)^2\\
& \le \left( z_1+\frac{\eps}{32}+z_2+\frac{\eps}{32} +z_3+\frac{\eps}{32}\right)^2 = \left(z + \frac{3}{32}\eps\right)^2 \le z^2+4\frac{3}{32}\eps=
z^2+\frac{3}{8}\eps.
\end{align*}
Observe that $z \ge 3/2 - 3\eps/16 \ge 1/2$, so $z^2 \ge 1/4$, and so, appealing again to Lemma \ref{l1},
\[ \left| \| u+v+w\|-\sqrt{4z^2+\frac{3}{8}\eps}\right| =
\left| \| u+v+w\|-\sqrt{\left|z^2+\frac{3}{8}\eps-(x^2+y^2)\right|+3z^2+x^2+y^2}\right| \le \frac{3}{8}\eps. \]
Hence 
\[ \| u+v+w\|\ge \sqrt{4z^2+\frac{3}{8}\eps}-\frac{3}{8}\eps\ge 2z-\frac{3}{8}\eps\ge 2\left(\frac{3}{2}-
\frac{3}{16}\eps\right)-\frac{3}{8}\eps>3-\eps.\]
\end{proof}

\begin{re}
Propositions \ref{prop-modulus-Ptolemy}, \ref{prop-modulus-DW} and \ref{prop-modulus-UBW'} provide an explicit modulus of uniform betweenness for  the classes of spaces discussed in this section. Thus, we can apply Theorem \ref{thm-main} and Corollary \ref{co-main} for the Lion-Man game played in bounded convex subsets of these spaces.
\end{re}

\section{Comments on the use of logic in arriving at the quantitative analysis (`proof mining')}

The point of departure for the investigation in this paper has been the 
noneffective proof for the convergence $\lim\limits_{n\to \infty} 
d(L_{n+1},M_n)=0$  
for {\it compact} uniquely geodesic spaces 
satisfying the betweenness property as given in \cite{LNP1} (Theorem 
4.2). Since the sequence $(d(L_{n+1},M_n))$ {\it decreases} to $0$ 
this statement is 
of the logical form 
\[ \forall k\in\N\,\exists n\in\N\, \left( d(L_{n+1},M_n)<2^{-k}\right)\in \forall\exists.\]
General logical metatheorems due to the first author (see, e.g., 
\cite{K08}) guarantee in such situations the extractability of an 
explicit and effective rate of convergence which only depends on general 
metric bounds, a modulus of total boundedness (as a quantitative form of 
the compactness assumption), and moduli providing quantitative forms of 
`uniformized' versions of being `uniquely geodesic' and satisfying the 
`betweenness property'. Technically speaking, these moduli serve to 
produce a solution for the monotone G\"odel functional interpretation (see \cite{K08}) of 
the respective properties which in this uniformized form become (essentially) 
purely universal assumptions in these moduli which can be taken as   
number-theoretic functions (although, for convenience, we used their 
$\varepsilon/\delta$-variants). Hence, the moduli can w.l.o.g. be assumed to be 
self-majorizing (in the technical sense of \cite{K08}) which would not be the 
case if these moduli would not be uniform by depending on points in $X.$

Subsequently it turned out that the uniqueness of geodesics (and hence the 
use of a modulus of uniform uniqueness in the quantitative analysis) could be 
avoided.

The actual extraction of the rate of convergence, given a modulus 
of uniform betweenness, from (the aforementioned generalization to the 
non-uniquely 
geodesic case of) the 
convergence proof of \cite[Theorem 4.2]{LNP1} makes this proof completely 
constructive by 
avoiding altogether the (in fact iterated and nested) 
sequential compactness argument used in the 
original proof and even the 
need to assume compactness in the first place (but only boundedness) 
{\it once} the betweenness property is written in its uniform variant 
(to which it is 
equivalent in the presence of compactness). As a consequence, 
the actual rate of convergence extracted 
only uses a modulus of uniform betweenness, but no modulus of total 
boundedness and holds for any bounded metric 
space satisfying the uniform betweenness property with the given modulus. 
Moreover, avoiding sequential compactness explains the very low complexity 
of the rate of convergence. By contrast, the use of sequential 
compactness, being equivalent to the principle of so-called arithmetic 
comprehension, see \cite{Sim99}, is known to in general require 
the very complex computational schema of so-called bar recursion for its 
functional interpretation (see \cite{K08}, chapter 11).
The phenomenon that any need to assume compactness altogether 
disappears from the proof in the 
process of its logical analysis is a feature of this particular proof 
being analyzed (see \cite{KLN} for situations where this is not 
the case).

Let us explain the nature of our uniformization of the betweenness property 
and the reason why the various moduli $\Theta$ of uniform betweenness could
be produced in all the cases considered in this paper from a logical point 
of view: the betweenness property can be logically re-written as  
\[ \ba{l} \forall x,y,z,w\in X\,\forall k,m\in\N\,\exists n\in\N \\[2mm] 
\left( \left. \ba{l} \sep\{ x,y,z,w\}\ge 2^{-k}  \\[1mm] 
d(x,y)+d(y,z)\le d(x,z) +2^{-n}  \\[1mm] 
d(y,z)+d(z,w)\le d(y,w)+2^{-n} \ea \right\}
\rightarrow 
d(x,z)+d(z,w)<d(x,w)+2^{-m}\right), \ea\]
where $\Big(\ldots\Big)$ is equivalent to a purely existential 
formula $A_{\exists}$ (since $\le,<$ between reals are purely 
universal resp. existential relations; see \cite{K08}).  

Logical bound extraction theorems (see 
e.g. in \cite{K08} Theorem 17.52 for the metric and Theorem 17.69
for the normed case) allow one to extract from (suitable) proofs of 
$X$ satisfying the betweenness property, bounds (and hence {\it 
realizers}) $\Theta(m,b,k)$ 
for $\exists n\in\N$ which only depend on $k,m$ and {\it majorants} for 
$x,y,z,w$ relative to some reference point $a\in X$ 
one can choose. For $a:=x$ this amounts to having a bound $b\in\N$ with 
$b\ge d(x,y),d(x,z),d(x,w)$ which (up to a factor $2$) 
is equivalent to having a $b$ with $b\ge\diam\{ x,y,z,w\}.$ 
$\Theta(n,k,b)$ is - when written more conveniently in $\eps/\delta$-notation - 
nothing else but our concept of a modulus for uniform betweenness.

In a similar way, one can extract a modulus of uniform 
uniqueness from suitable proofs of ordinary uniqueness.

By `suitable' proofs we mainly mean that the space $X$ 
in question either is a concretely definable boundedly compact space 
(such as $(\R^3,\|\cdot\|_{\rm DW})$ from the previous section) or - 
more importantly - the proof only uses that $X$ 
belongs to a general class of `abstract' spaces. 
In the case of our paper, the relevant classes of spaces are  
metric spaces, uniformly 
convex geodesic and normed spaces (with given modulus), 
geodesic spaces, uniformly uniquely geodesic spaces 
(with given modulus) satisfying the convexity condition 
(\ref{eq-metric-conv}) and Ptolemy spaces 
which all are permitted in (suitable adaptations) of the aforementioned 
metatheorems. For metric spaces, uniformly convex geodesic as well as 
uniformly convex normed spaces such metatheorems are explicitly formulated 
in the literature. For Ptolemy spaces one only has to observe that the 
defining axiom is purely universal (and any such assumption is permitted).
This also applies to the case of uniformly uniquely geodesic spaces (with 
given modulus) with a convex metric 
which can be axiomatized as in \cite{K08} by a triple $(X,d,W)$ 
where now, however, $W$ is only required to satisfy the axioms $(i),(ii)$ 
in Definition 17.9 in \cite{K08} plus another universal axiom 
stating that the given modulus satisfies the uniform uniqueness condition 
for the geodesics. Note that it is only due to the presence of 
this uniqueness axiom 
that the axiom $(i)$ expresses the convexity of the metric 
(which speaks about arbitrary geodesics and not just the one selected by 
$W$).  

All this explains on general logical grounds why the moduli
of uniform uniqueness (for geodesics) 
and uniform betweenness produced in this paper could be 
extracted from the resp. proofs of ordinary uniqueness and betweenness 
(provided that assumptions such as `strict convexity' had been properly 
uniformized to uniform convexity). 

Such uniformizations (and extractions of the resp. moduli) are 
systematically performed by the logical {\it proof-interpretations} 
(see \cite{K08}) used to 
establish the bound extraction theorems, namely modern 
monotone versions of G\"odel's so-called functional (`Dialectica') 
interpretation. In fact, the monotone functional interpretation allows 
one to eliminate (and hence to use freely) a {\it nonstandard} 
uniform boundedness principle $\exists$-UB$^X$ (see \cite{K08}, 
section 17.7 and, 
in particular, Theorem 17.101) which can 
easily be seen to imply the equivalence 
of ordinary uniqueness (of geodesics) and uniform uniqueness (with 
modulus) and of betweenness and uniform betweenness (with modulus) for 
any bounded geodesic resp. metric space (see section 17.8 in 
\cite{K08} for many similar results). Model-theoretically speaking this 
roughly corresponds (noneffectively) 
to going to an ultrapower of the structures in 
question (see e.g. \cite{GK16}).

As usual with case studies in proof mining, when the actual extraction of 
the data in question is carried out it also comes with an ordinary 
analytic proof of their correctness which does not refer to any 
results from logic which, however, were instrumental for finding these data.

\section{Acknowledgements}
Part of this work was carried out while the authors visited the Mathematical Research Institute of Oberwolfach (Research in Pair stay 1911p) and the Institute of Mathematics of the University of Seville - IMUS. They would like to thank these institutions for the support. This work was also partially supported by DGES (Grant MTM2015-65242-C2-1P), DFG (Project KO 1737/6-1), and by a grant of the Romanian Ministry of Education and Research, CNCS - UEFISCDI, project number PN-III-P1-1.1-TE-2019-1306, within PNCDI III.

\end{document}